\documentclass[11pt,a4paper]{article}
\usepackage{graphicx}
\usepackage{amsmath,amsfonts,amsthm}
\usepackage{xfrac,esint,mathrsfs,color}
\usepackage[utf8]{inputenc}
\usepackage[colorlinks=true]{hyperref}

\newcommand{\Huno}{{\calH}^{1}}      
\newcommand{\Rdue}{\R^2}

\newcommand{\parametermu}{\mu}

\newcommand{\C}{\mathbb{C}}
\newcommand\res{\mathop{\hbox{\vrule height 7pt width .5pt depth 0pt
\vrule height .5pt width 6pt depth 0pt}}\nolimits}
\newcommand{\Ln}{{\calL}^2}

\newcommand\eps{\varepsilon}

\newcommand\loc{\mathrm{loc}}

\newcommand\Tr{\mathrm{Tr\,}}
\newcommand\dist{\mathrm{dist}}

\newcommand\R{\mathbb{R}}
\newcommand\N{\mathbb{N}}

\newcommand\calH{\mathcal{H}}
\newcommand\calF{\mathcal{F}}

\newcommand\calA{\mathcal{A}}

\newcommand\calL{\mathcal{L}}
\newcommand{\LM}[1]{\hbox{\vrule width.2pt \vbox to#1pt{\vfill \hrule width#1pt height.2pt}}}
\newcommand{\LL}{{\mathchoice{\,\LM7\,}{\,\LM7\,}{\,\LM5\,}{\,\LM{3.35}\,}}}
\newcommand{\ba}[1]{\begin{eqnarray} #1 \end{eqnarray}}
\newcommand{\be}[1]{\begin{equation} #1 \end{equation}}

\newcommand{\fz}{f_0}

\newtheorem{theorem}{Theorem}[section]
\newtheorem{definition}[theorem]{Definition}
\newtheorem{proposition}[theorem]{Proposition}
\newtheorem{lemma}[theorem]{Lemma}

\newtheorem{remark}[theorem]{Remark}
\newtheorem{corollary}[theorem]{Corollary}

\numberwithin{equation}{section}
\newcounter{Nummer}

\usepackage{fancyhdr}
\newcommand{\Rn}{{\R}^n}
\newcommand{\cost}{\kappa}

\begin{document}
\begin{center}
  {\Large
{Existence of strong minimizers} for the Griffith static fracture model 
in dimension two}\\[5mm]
{\today}\\[5mm]
Sergio Conti$^{1}$, Matteo Focardi$^{2}$, and Flaviana Iurlano$^{3}$\\[2mm]
{\em $^{1}$
 Institut f\"ur Angewandte Mathematik,
Universit\"at Bonn\\ 53115 Bonn, Germany}\\[1mm]
{\em $^{2}$ DiMaI, Universit\`a di Firenze\\ 50134 Firenze, Italy}\\[1mm]
{\em $^{3}$ Laboratoire Jacques-Louis Lions, Université Paris 6\\ 75005 Paris, France}\\[3mm]
    \begin{minipage}[c]{0.8\textwidth}
 We consider the Griffith fracture model in two spatial dimensions, and prove
    existence of strong minimizers, with closed jump set  and  continuously
    differentiable deformation fields. One key ingredient, {which is the object of the present paper}, is 
    a generalization of the decay estimate by De Giorgi, Carriero, and Leaci to the
    vectorial situation. {This is} based on replacing the coarea formula by a method to approximate
 $SBD^p$ functions with small jump set by Sobolev functions {and is} restricted to two dimensions. {The other two ingredients are contained in companion papers and consist respectively in regularity results for vectorial elliptic 
 problems of the elasticity type and in a method to approximate in energy $GSBD^p$ functions by $SBV^p$ ones.} 
    \end{minipage}
\end{center}

\section{Introduction}
The study of brittle fracture in solids is based on the Griffith model, which combines
elasticity with a term proportional to the surface opened by the fracture. In its variational formulation one minimizes
\begin{equation}\label{eqgriffintro}
 E[\Gamma,u]:=\int_{\Omega\setminus\Gamma} 
 \Big(\frac12 \C e(u)\cdot e(u)+h(x,u)\Big) dx + 2\beta\calH^{n-1}({\Gamma\cap\Omega})
\end{equation}
over all closed 
sets $\Gamma\subset{\overline\Omega}$ and all deformations $u\in C^1(\Omega\setminus\Gamma,{\R^n})$ subject to suitable
boundary and irreversibility conditions. Here $\Omega\subset\R^n$ is the reference configuration, the function $h\in C^0(\Omega\times \R^n)$ represents external volume forces,
$e(u)=(\nabla u+\nabla u^T)/2$ is the elastic strain,
$\C\in \R^{(n\times n)\times (n\times n)}$ is the matrix of elastic coefficients, $\beta>0$ the surface energy.
The evolutionary problem of fracture can be modeled as a sequence of variational problems, in which
one minimizes \eqref{eqgriffintro} subject to varying loads with a kinematic restriction representing
the irreversibility of fracture, see \cite{FrancfortMarigo1998,BourdinFrancfortMarigo2008,dm-toa}.

Mathematically, \eqref{eqgriffintro} is a vectorial free discontinuity problem. Much better known is its scalar 
version, mechanically corresponding to the anti-plane case, in which one
replaces the elastic energy by the Dirichlet integral,
\begin{equation}\label{eqgMSintro}
 E_\mathrm{MS}[\Gamma,u]:=\int_{\Omega\setminus\Gamma} \Big(\frac12 |Du|^2+{h(x,u)}\Big) dx + 2\beta\calH^{n-1}(\Gamma{\cap\Omega}) \,,
\end{equation}
and one minimizes over all maps $u:\Omega\setminus\Gamma\to\R$. This scalar reduction coincides with the Mumford-Shah 
functional of image segmentation, and has been widely studied analytically and numerically
\cite{AmbrosioFuscoPallara,David2005,BourdinFrancfortMarigo2008}. The relaxation of 
\eqref{eqgMSintro} leads naturally to the space of special functions of bounded variation, and is given by
\begin{equation}\label{eq:grifsbv}
 E_\mathrm{MS}^*[u]:=\int_{\Omega} \Big(\frac12 |\nabla u|^2 +{h(x,u)}\Big)dx + 
 2\beta \calH^{n-1}(J_u{\cap\Omega}) \,.
\end{equation}
Here $u$ belongs to the space $SBV^2(\Omega)$, which is the set of  functions such that the distributional gradient $Du$ is a bounded measure and can be written as 
$Du=\nabla u \calL^n + [u]{\nu_u} \calH^{n-1}\LL J_u$ with $\nabla u\in L^2(\Omega;\R^n)$, $[u]$ the jump of $u$, 
$J_u$ the {$(n-1)$}-rectifiable jump set of $u$, which obeys $\calH^{n-1}(J_u)<\infty$, and ${\nu_u}$ its normal. Existence of minimizers for the relaxed
problem $E_\mathrm{MS}^*$ follows then from the general compactness properties of $SBV^2$, see \cite{AmbrosioFuscoPallara} 
and references therein.

The breakthrough in the quest for an existence theory for
 the Mumford-Shah functional \eqref{eqgMSintro}
came with the proof by De Giorgi, Carriero, and Leaci in 1989 \cite{DegiorgiCarrieroLeaci1989}
that the jump set of minimizers is essentially closed, in the sense that
 minimizers of the relaxed functional $ E_\mathrm{MS}^*$ obey
\begin{equation}\label{eqjuclosedintro}
\calH^{n-1}(\Omega\cap J_u)=\calH^{n-1}({\Omega\cap\overline {J_u}}).  
\end{equation}
This permits to define $\Gamma$ as the closure
of $J_u$, and then to use regularity of local minimizers of the Dirichlet integral on the open set 
$\Omega\setminus\Gamma$ to prove smoothness of $u$. 
{The essential closedness of the jump set stated in \eqref{eqjuclosedintro} is a property 
satisfied by several variants of the energy in \eqref{eq:grifsbv}, in particular also by some 
defined on vector-valued $SBV^2(\Omega,\R^N)$ functions.
More precisely, the integrands dealt with in literature depend on the full gradient 
with some additional structure conditions: they are either convex and depending (essentially) 
on the modulus of the gradient (cf. \cite{CarrieroLeaci91,FonsecaFusco97,FuscoMingioneTrombetti})
or they are specific polyconvex integrands in two dimensions, i.e.~$n=2$
(cf. \cite{AcerbiFonsecaFusco_SNS97,AcerbiFonsecaFusco_Edin97}).}

In this paper we study existence for \eqref{eqgriffintro} in two spatial dimensions;
{therefore the main difference with the  results quoted above is the dependence of the bulk 
energy density on the linear elastic strain rather than on the full deformation gradient. Indeed,}
we assume that $\C$ is a symmetric linear map from $\R^{n\times n}$ to itself with the properties
\begin{equation}\label{eqassC}
 \C(\xi-\xi^T)=0 \text{ and } \C \xi\cdot \xi \ge c_0 |\xi+\xi^T|^2 \text{ for all } \xi\in\R^{n\times n}.
\end{equation}
This includes of course as a special case isotropic elasticity, 
$\C\xi\cdot \xi = \frac14\lambda_1|\xi+\xi^T|^2+ \frac12\lambda_2 (\Tr\xi)^2$, where $\lambda_1$ and $\lambda_2$ are the Lam\'e constants.

Our main result is the following.
\begin{theorem}\label{theop2}
 Let $\Omega\subset\R^2$ be a bounded Lipschitz set, {$g\in 
 L^\infty(\Omega;\R^2)$}, let $\C$ obey the positivity condition \eqref{eqassC}, {$\beta>0$}, {$h(x,z):=\cost|z-g(x)|^2$ for some $\cost>0$.}
 Then the functional \eqref{eqgriffintro} has a minimizer in the class
 \begin{equation*}
  \calA:=\{(u,\Gamma): \Gamma\subset\overline \Omega \text{ closed, } u\in C^1(\Omega\setminus\Gamma;\R^2)\}.
 \end{equation*}
\end{theorem}
This result was announced in \cite{ContiFocardiIurlano2016-CRAS}. 

We also consider a generalization of the basic model \eqref{eqgriffintro} with $p$-growth, {which may be appropriate
for the study of fracture models with nonlinear constitutive relations that account for damage and plasticity, 
see for example \cite[Sect. 10 and 11]{hutchinson1989course} and references therein.}
We replace the quadratic energy density and the lower order term by the functions
\begin{eqnarray}\label{eqdeffintro}
& f_\parametermu(\xi):=\frac 1p\left(\big(\mathbb{C}\xi\cdot\xi+\parametermu\big)^{\sfrac p2}-\parametermu^{\sfrac p2}\right),\\
& {h(x,z):=\cost |z-g(x)|^p,}\nonumber
\end{eqnarray}
where $\parametermu\geq 0$ {and $\cost>0$ are} parameters and {$g\in L^\infty(\Omega;\R^2)$}.
We remark that for $\parametermu>0$ and for small {strains} $\xi$ this energy reduces to linear 
elasticity, $f_\parametermu(\xi)=\frac12 \parametermu^{p/2-1} \C\xi\cdot \xi + O(|\xi|^3)$. For large $\xi$ it behaves, 
up to multiplicative factors, as $|\xi+\xi^T|^p$, which is for example appropriate for models that describe plastic deformation at large strains. We obtain the following. 
\begin{theorem}\label{theootherp}
 Let $\Omega\subset\R^2$ be a bounded Lipschitz set, $p\in (1,\infty)$, 
 {$\mu\ge0$,} {$\kappa,\beta>0$,} {$g\in L^\infty(\Omega;\R^2)$ 
 if $p\in(1,2]$ and  $g\in W^{1,p}(\Omega;\R^2)$ if $p\in(2,\infty)$}, let $\C$ 
 obey the positivity condition \eqref{eqassC}, and 
 let $f_\parametermu$ be as in \eqref{eqdeffintro}.
 Then the functional
 \begin{equation}\label{eqgriffintrop}
 {E_p[\Gamma,u]:=\int_{\Omega\setminus\Gamma} (f_\parametermu(e(u))+\cost|u-g|^p) dx 
 +2\beta \calH^{1}(\Gamma\cap\Omega)} 
\end{equation}
has a minimizer in the class
 \begin{equation}
  \calA_p:=\{(u,\Gamma): \Gamma\subset\overline \Omega \text{ closed, } u\in C^1(\Omega\setminus\Gamma;\R^2)\}.
 \end{equation}
\end{theorem}
\begin{remark}
{The assumption $g\in W^{1,p}(\Omega;\R^2)$ if $p>2$ is probably 
of technical nature and depends on the elliptic regularity results {discussed} in Section~\ref{secregularity}.}
\end{remark}

In the last years several approaches have been proposed to show existence for 
$E_\mathrm{MS}$ after the seminal paper by De Giorgi, Carriero, and Leaci \cite{DegiorgiCarrieroLeaci1989} in which the result
has been first established (cp. \cite{CarrieroLeaci91, FonsecaFusco97, DalMasoMorelSolimini, MaddalenaSolimini, David2005, DeLellisFocardi, BucurLuckhaus}, 
and \cite{FocardiReview}  
for a recent review). Here we follow 
the general strategy of proof by De Giorgi, Carriero, and Leaci \cite{DegiorgiCarrieroLeaci1989}, 
although several new difficulties inherent to 
{the dependence of the bulk energy density on the symmetrized gradient} have to be faced. 

We start {off} writing the relaxed formulation of \eqref{eqgriffintro}, which 
{for $\kappa>0$} has
a minimizer in the space {$GSBD^p(\Omega)$ since no $L^\infty$ bound is imposed (see below for the precise definition of the functional setting). 
This space and its companion $SBD^p$ are, however, much less understood than the scalar analogues $(G)SBV^p$, though in the last few years there have been several contributions in this direction \cite{gbd, ChambolleContiFrancfort2016, ContiFocardiIurlano2016, ContiFocardiIurlanoRepr, ContiFocardiIurlanoGBD, Friedrich1, Friedrich2, FriedrichSolombrino}. 
In particular, since apart from trivial cases the Chain rule formula does not hold in $SBD^p$, the very definition of the generalized space $GSBD^p$ given in \cite{gbd} requires a different approach with respect 
to the standard definition of $GSBV^p$ as the set of functions whose truncations belong to $SBV^p$. 

The proof given in \cite{DegiorgiCarrieroLeaci1989} of the closure condition \eqref{eqjuclosedintro} 
in the scalar case is based on a careful analysis of sequences of $SBV^p$ (quasi-)minimizers
with vanishing jump energy, for which a priori no control of any Lebsgue norm is available.
The idea to circumvent this difficulty and to gain compactness in $SBV^p$ introduced} 
by De Giorgi, Carriero, and Leaci, however, makes substantial use of a Poincar\'e-type inequality for $SBV$ functions that is proven via the coarea formula, which does not extend to 
the vectorial case.
One key ingredient in our proof is then an approximation result for $SBD^p$ functions 
with small jump set with $W^{1,p}$ functions, stated in {Proposition~\ref{prop:ricopr} 
below, which permits to obtain an equivalent Poincar\'e-type inequality for $SBD^p$ functions, 
however restricted to two spatial dimensions (see \cite{ContiFocardiIurlanoRepr} for the proof).
We remark explicitly that this is the only issue in which we have to confine to two dimensions. 
Indeed, the other two key results of our approach have higher dimensional analogues.
More precisely, for $n\geq 3$ the full elliptic regularity of solutions to linear elasticity 
type systems stated in Theorem~\ref{t:regularity} has a partial regularity counterpart 
with an estimate on the Hausdorff dimension of the singular set (see for more details the 
comments in Section~\ref{secregularity}), and the strong approximation result of $GSBD^p$ 
functions with $SBV^p\cap L^\infty$ ones in Theorem~\ref{flaviana} holds without any 
dimensional limitation. 
Therefore an extension of the Poincar\'e-type inequality for $SBD^p$ functions to higher dimensions, 
would {lead to corresponding generalizations of} Theorems~\ref{theop2} and \ref{theootherp}, {at least for $p=2$}.

Going back to commenting the proof, we note that} rather than extending the 
{quoted Poincar\'e-type inequality for $SBD^p$ functions} to $GSBD^p$ {ones}, we argue 
by approximating $GSBD^p$ functions by $SBD^p$ ones in energy. The latter issue is discussed in 
\cite{Iurlano2014} for $p=2$ and any dimension, see Section~\ref{prelim} below. 
The case of a general exponent $p\in(1,\infty)$ {is} established in a companion paper 
\cite{ContiFocardiIurlanoGBD} {without dimensional restrictions and requires a nontrivial modification of
the original arguments in \cite{Chambolle2004a,Chambolle2004b,Iurlano2014}.}
Since the $SBD^p\mbox{-}GSBD^p$ approximation does not preserve the boundary values, one additionally needs to suitably combine the two approximation results {carefully.

Let us also stress that under the working assumption that $g$ is bounded, by the maximum principle, 
i.e.~by truncations, the fidelity term in the scalar case is a lower order perturbation that originates and justifies 
the more general regularity theory developed in literature for Mumford-Shah quasi-minimizers. 
In the vector valued setting {of interest here} instead, for the above mentioned lack of truncation 
techniques, such a term plays a nontrivial role in the asymptotic analysis of sequences with infinitesimal jump energy 
and has to be taken into account (cf. Proposition~\ref{propconvenerg}).

In any case, the asymptotics of such sequences in the framework under investigation 
is related, similarly to the scalar setting,} 
to minimizers of an elliptic problem. In the scalar case, standard elliptic regularity directly 
gives the necessary 
decay estimates for the energy {(cf. \cite{AmbrosioFuscoPallara,  FonsecaFusco97})}. 
{The case of the system of linearized elasticity is also well-known in literature.
Instead, for systems of linearized elasticity type with $p\neq 2$} the regularity is less standard, 
and we summarize the results we need in Section~\ref{secregularity}. 
{Details and extensions to higher dimensions are discussed elsewhere \cite{ContiFocardiIurlanoRegularity}.
In particular, partial regularity with an explicit estimate on the Hausdorff dimension of the potential 
singular set are established in \cite{ContiFocardiIurlanoRegularity}. 
We remark that it is a major open problem to prove or disprove full regularity in the case $p\neq 2$. 
Despite this, the mentioned Hausdorff dimension estimate is particularly relevant in view of the 
possible extensions of the existence of minimizers of the energy in \eqref{eqgriffintro} in higher dimensions}.


{
Our {main contribution} is a statement on the regularity of weak 
{local minimizers (cf. \eqref{e:Glocmin} for the precise definition)}. 
In particular, we show (see Theorem \ref{t:theogsbdpmin} below) that if {$u\in GSBD^p(\Omega)$} 
 is a local minimizer for the weak formulation then 
 $\calH^1(\Omega\cap \overline {J_u}\setminus J_u)=0$ and $u\in C^1(\Omega\setminus \overline{J_u};\R^2)$.
The condition $\kappa>0$ is only required for establishing the existence of a weak minimizer in $GSBD^p(\Omega)$ via 
\cite[Theorem 11.3]{gbd}.}

{Let us conclude the introduction by outlining the organization of the paper. We first provide the technical 
preliminaries: in Section~\ref{secregularity} we state the needed elliptic decay estimates, then in 
Section~\ref{prelim} we introduce the spaces $SBD^p$ and $GSBD^p$ and discuss the quoted approximation results.
In Section~\ref{s:dlb} we prove the density lower bound and essential closedness of the jump set for local minimizers.
Finally, in Section~\ref{mainproof} we prove the main results Theorem~\ref{theop2} and Theorem~\ref{theootherp}.}

\section{Preliminaries}

\subsection{Regularity for generalized linear elasticity systems}\label{secregularity}

In this section we investigate the regularity properties of minimizers of elastic type energies.
Despite several related contributions present in literature (see \cite{FuchsSeregin,DieningEttwein,DieningKaplicky13} 
and references therein), we have not found the exact statements needed for our purposes. 
We summarize here the results of interest, and {provide} elsewhere \cite{ContiFocardiIurlanoRegularity} 
a self-contained proof of the elliptic decay estimates {as well as of full and partial regularity 
for local minimizers according to the dimensional setting of the problem,} following the techniques of 
\cite{AcerbiFusco89,DieningEttwein,DieningKaplicky13,GiaquintaMartinazzi2012,Giusti}. 

We first present a decay property of the $L^p$-norm of $e(u)$, with $u$ {a local} minimizer of 
$v\mapsto\int_\Omega f_0(e(v))dx$, 
{i.e.,
\[
 \int_\Omega f_0(e(u))dx\leq \int_\Omega f_0(e(v))dx,
\]
for all $v\in W^{1,p}(\Omega;\R^n)$ satisfying $\{v\neq u\}\subset\subset\Omega$.
Such a result is necessary} to prove the density lower bound inequality in Section~\ref{s:dlb}. 
Since {in this paper the decay property} will be applied to the blow-ups of minimizers, 
there are no lower order terms,  
{therefore we state the result only for the functional with $\kappa=\parametermu=0$
(cf. \cite{ContiFocardiIurlanoRegularity} for the proof given in the general case)}.
\begin{proposition}[{\cite[Proposition~3.4]{ContiFocardiIurlanoRegularity}}]\label{p:decayVmu}
Let {$n=2$, $\Omega\subset\R^2$} open, $p\in(1,\infty)$.
Let $u\in W^{1,p}(\Omega;\R^2)$ be a local minimizer of 
\[
v\mapsto\int_\Omega f_0(e(v))dx.
\] 
Then, for all $\gamma\in(0,2)$ there is a 
constant ${c_\gamma}=c(\gamma,\mathbb{C},p)>0$ such that 
{for all $\rho<R\leq 1$ such that 
$B_{R}(x_0)~\subset\hskip-0.125cm\subset\Omega$} it holds
\begin{equation*}
 \int_{B_{\rho}(x_0)}{f_0(e(u))}dx\leq {c_\gamma}\,\left(\frac{\rho}R\right)^{2-\gamma} 
 \int_{B_{R}(x_0)}{f_0(e(u))}dx\,.
\end{equation*}
\end{proposition}

In the quadratic case $p=2$ it is well-known that the minimizer 
$u$ is $C^\infty(\Omega;\R^n)$ in any dimension {as long as $g$ is smooth} (see for instance \cite[Theorem~10.14]{Giusti} 
or \cite[Theorem~5.14, Corollary~5.15]{GiaquintaMartinazzi2012}). 
Below we {state a $C^{1,\alpha}$ regularity result} in the two dimensional setting ({ see \cite[Section 4]{ContiFocardiIurlanoRegularity}} for the proof and for extensions to higher 
dimensions).
\begin{theorem}[{\cite[Proposition~4.3]{ContiFocardiIurlanoRegularity}}]\label{t:regularity}
Let $n=2$, $\Omega\subset\R^2$ open, $p\in(1,\infty)$ {$\cost\ge0$} and $\parametermu\geq 0$, 
$g\in W^{1,p}(\Omega;\R^2)$ if $p>2$, {$g\in L^\infty(\Omega;\R^2)$} if $p\in(1,2]$. 
Let $u\in W^{1,p}(\Omega;\R^2)$ be a local minimizer of 
\[
v\mapsto\int_\Omega f_\parametermu(e(v))dx{+\cost\int_\Omega|v-g|^pdx}.
\] 
Then, $u\in C^{1,\alpha}_\loc(\Omega;\R^2)$ for all $\alpha\in(0,1)$ if $\parametermu>0$, and for 
some $\alpha(p)\in(0,1)$ if $\parametermu=0$.
\end{theorem}

\subsection{Approximation of  \texorpdfstring{$SBD^p$}{SBDp} and  \texorpdfstring{$GSBD^p$}{GSBDp} functions}
\label{prelim}
We start by briefly  collecting the main properties of $GBD$ and $GSBD^p$ of interest to us.
Let $\Omega$ be a bounded open set in $\R^n$. 
If $u:\Omega\to\R^n$ is a Borel function, we say that $x\in\Omega$ is a point of approximate continuity for $u$ if there is $a\in\R^n$ such that for each $\varepsilon>0$
\begin{equation}\label{eqapprcont}
\lim_{r\to0}\frac{1}{r^n}\calL^n\left(B_r(x)\cap\{|u-a|\geq\varepsilon\}\right)=0. 
\end{equation}
We say that $x$ is a jump point, and we write $x\in J_u$, if there exist two distinct vectors $a^{\pm}\in \R^n$ and a unit vector $\nu\in \R^n$ such that
the approximate limit of the restriction of $u$ to $\{y\in\Omega:\ \pm(y-x)\cdot \nu>0\}$ is $a^{\pm}$. 

The space $BD(\Omega)$ of functions with bounded deformation in $\Omega$ and its subspace $SBD(\Omega)$ have been widely studied due to their role in the variational formulation of many problems in plasticity and fracture mechanics. Let us recall that the jump set $J_u$ of a function $u\in BD(\Omega)$ is countably $(\calH^{n-1},n-1)$-rectifiable and that for $\calH^{n-1}$-a.e. $x\in J_u$ the function $u$ has one-sided approximate limits $u^\pm(x)$ with respect to a suitable direction $\nu_u(x)$ normal to $J_u$ at $x$.
{We denote by $S_u$ the set of approximate discontinuity points, in the sense of the set of points where \eqref{eqapprcont} does not hold.}
Moreover one can define the approximate symmetric gradient $e(u)\in L^1(\Omega;\R^{n{\times} n})$. For further details and properties see \cite{temam,AmbrosioCosciaDalmaso1997,Babadjan2015,DephilippisRindler,gbd}.

The subspace $SBD^p(\Omega)$, $p>1$, contains all functions $u\in BD(\Omega)$ whose symmetric distributional derivative can be decomposed as 
\[
Eu=e(u)\calL^n\res\Omega
+(u^+-u^-)\odot\nu_u\calH^{n-1}\res J_u, 
\]
with $e(u)\in L^p(\Omega;\R^{n{\times} n})$ and $\calH^{n-1}(J_u)<\infty$. 
Fine properties and rigidity properties of $SBD^p$ have been highlighted in \cite{BellettiniCosciaDalmaso1998,Chambolle2004a,cha-gia-pon,ChambolleContiFrancfort2016,Friedrich1,Friedrich2,ContiFocardiIurlano2016}.

The generalized space $GSBD(\Omega)$ introduced in \cite{gbd} has proved to be the correct space where setting a number of problems in linearized elasticity, see \cite{Iurlano2014,FriedrichSolombrino}. 
An $\mathcal{L}^n$-measurable function $u\colon\Omega\to\mathbb{R}^n$ belongs to $GSBD(\Omega)$ if there exists {a bounded positive Radon measure} $\lambda_u\in\mathcal{M}_b^+(\Omega)$
such that the following condition holds for every $\xi\in\mathbb{S}^{n-1}$: 
for $\mathcal{H}^{n-1}$-a.e.\ $y\in\Omega^\xi$ the function $u^\xi_y$ 
defined by {$u^\xi_y(t):=u(y+t\xi)\cdot\xi$}
belongs to $SBV_\loc(\Omega^\xi_y)$, where $\Omega^\xi_y:=\{t\in\R: y+t\xi\in\Omega\}$, and for every Borel set $B\subset \Omega$ it satisfies
\begin{equation}\label{e:sl}
\int_{{\Omega}^\xi}\Big(|Du^\xi_y|(B^\xi_y\setminus J^1_{u^\xi_y})+\mathcal{H}^0(B^\xi_y\cap J^1_{u^\xi_y})\Big)d\mathcal{H}^{n-1}\leq
\lambda_u(B),
\end{equation}
where $J^1_{u^\xi_y}:=\{t\in J_{u^\xi_y}:|[u^\xi_y](t)|\geq1\}$.

If $u\in GSBD(\Omega)$, the aforementioned quantities $e(u)$ and $J_u$ are still well-defined, and are respectively integrable and rectifiable in the previous sense. In analogy to $SBD^p(\Omega)$, the subspace $GSBD^p(\Omega)$ includes all functions in $GSBD(\Omega)$ satisfying $e(u)\in L^p(\Omega;\R^{n{\times} n})$ and $\calH^{n-1}(J_u)<\infty$.

Next proposition states that a $GSBD^p$-function with a small jump set can be approximated by Sobolev functions. 
This is a minor reformulation of the result of  \cite{ContiFocardiIurlanoRepr} 
(see \cite{ChambolleContiFrancfort2016,Friedrich1,Friedrich2} for related works in $SBD^p$). 
Its proof is based on first covering the jump set with countably many balls with finite overlap and properties
(i) and (ii), and then in each ball $B$  constructing $w$ as a piecewise affine approximation to $u$ on a 
suitably chosen triangular grid, which refines towards $\partial B$ in such a way that grid segments do not 
intersect $J_u$, following a strategy developed in \cite{ContiSchweizer2}.

\begin{proposition}\label{prop:ricopr}
Let $p\in(1,\infty)$, $n=2$. 
There exist universal constants $c,\eta, {\xi}>0$ such that if $u\in SBD^p(B_{\rho})$, {$\rho>0$,} satisfies
\[
\Huno(J_u\cap B_{\rho})<{\eta\,(1-s)\frac\rho2}
\]
for some $s\in(0,1)$, then there are a 
countable family $\calF=\{B\}$ of closed balls {overlapping at most $\xi$ times} 
of radius $r_B<(1-s)\rho/2$ and center $x_B\in \overline B_{s\rho}$, 
and a field $w\in SBD^p(B_{\rho})$  such that
\begin{enumerate}
\item {$\eta r_B\leq \calH^1(J_u\cap B)\leq 2\eta r_B$} {for all $B\in\calF$;}
\item $\Huno\big(J_u\cap\cup_{\calF}\partial B\big)=\Huno\big((J_u\cap {B_{s\rho}})\setminus \cup_{\calF}B\big)=0$;
\item $w=u$ {$\Ln$-a.e.} on $B_{\rho}\setminus\cup_{\calF}B$;
\item {$w\in W^{1,p}(B_{s\rho};\Rdue)$} and $\Huno(J_w\setminus J_u)=0$;
\item for each $B\in\calF$ one has { $w\in W^{1,p}(B;\R^2)$ with}
\begin{equation}\label{e:volume}
 \int_{B}|e(w)|^pdx\leq c\int_{B}|e(u)|^pdx; 
\end{equation}
{and there exists a skew-symmetric matrix $A$ such that 
\begin{equation}\label{e:korn}
 \int_{B_{s\rho}\setminus\cup_{\calF}B}|\nabla u-A|^pdx\leq c\int_{B_{\rho}}|e(u)|^pdx; 
\end{equation}}
\item
$\cup_\calF B \subset {B_{\frac{1+s}{2}\rho}}$ and 
{$\sum_\calF\calL^2(B)\leq \frac c\eta\,\rho\,\calH^1(J_u\cap B_\rho)$};
\item if, additionally, $u\in L^\infty(B_{\rho};\Rdue)$ then $w\in L^\infty(B_{\rho};\Rdue)$ with
\[
\|w\|_{L^\infty(B_{\rho};\Rdue)}\leq \|u\|_{L^\infty(B_{\rho};\Rdue)}.
\]
\end{enumerate}
\end{proposition}
The next result is an approximation in energy of $GSBD^p$ functions with $SBV^p$ functions, 
which was proven in \cite{Iurlano2014} for $p=2$ and for any dimension, building upon ideas 
developed in  \cite{Chambolle2004a,Chambolle2004b} for $SBD^2$ functions.
The extension to $p\ne 2$ {is discussed in details} elsewhere \cite{ContiFocardiIurlanoGBD}. 
Let us only mention that despite we still follow the ideas in \cite{Chambolle2004a,Chambolle2004b}, 
in the nonquadratic case a different definition of the piecewise affine approximants is needed. 
{Indeed, it requires the use of a different interpolation scheme and a different finite-element 
grid for the actual construction.} 
\begin{theorem}[{\cite[Theorem~3.1]{ContiFocardiIurlanoGBD}}]\label{flaviana}
 Let $\Omega\subset\R^n$ be a bounded Lipschitz set, 
 ${u\in GSBD^p(\Omega)}\cap {L^p(\Omega;\R^n)}$. 
 Then there is a sequence $v_j\in L^\infty\cap SBV^p(\Omega;\R^n)$ 
 such that 
 \begin{equation*}
  \lim_{j\to\infty}\big(\|e(v_j)-e(u)\|_{L^p(\Omega;\R^{n\times n})} +\|v_j-u\|_{L^p(\Omega;\R^n)} 
  +{\left|{\calH^{n-1}(J_{v_j})-\calH^{n-1}(J_u)}\right|} \big)=0\,.
 \end{equation*}
\end{theorem}

\section{Proof of existence of strong minimizers}\label{s:dlb}
We prove that weak minimizers (in $GSBD^p$) have an essentially closed jump set,
and therefore can be identified with strong minimizers. The general strategy is similar to the
one by De Giorgi, Carriero, and Leaci \cite{DegiorgiCarrieroLeaci1989}; the key new ingredients
are the approximation results for $GSBD^p$ functions with Sobolev functions discussed in Section \ref{prelim}
{ and corresponding rigidity estimates for treating the lower-order term.}

\subsection{Density lower bound}
In this section we assume that {$\cost\ge0$}, $\beta>0$, $p>1$, $g\in L^\infty(\Omega;\R^{n})$, 
$\parametermu\geq 0$  are given, {and that $\Omega\subset\R^{n}$ is a bounded, open, Lipschitz set}. For all $u\in GSBD(\Omega)$ and all {Borel sets} $A\subset\Omega$ 
we define the functional
\begin{equation}\label{e:G} G(u,\cost,\beta,A):=\int_A f_\parametermu(e(u))dx +\cost\int_A|u-g|^pdx
+ 2\beta \calH^{n-1}(J_u\cap {A}).
\end{equation}
By \cite[Theorem 11.3]{gbd} the global minimum problem for $G$ has a solution in 
$GSBD(\Omega)$. 
{Moreover, we say that $u\in GSBD^p(\Omega)$ is a local minimizer of $G(\cdot,\cost,\beta,\Omega)$ provided
\begin{equation}\label{e:Glocmin}
 G(u,\cost,\beta,\Omega)\leq  G(v,\cost,\beta,\Omega),
\end{equation}
for all $v\in GSBD^p(\Omega)$ satisfying $\{v\neq u\}\subset\subset\Omega$.}

In order to prove the main result of the paper, Theorem \ref{theootherp}, 
we use an homogeneous version of $G$   
\[ 
G_0(u,\cost,\beta,A):=\int_A f_0(e(u))dx +\cost\int_A|u|^pdx+ 
2\beta \calH^{n-1}(J_u\cap {A})
\]
to get an appropriate decay estimate (Lemma \ref{lemmadecay}) and then density lower bounds 
{for the full energy $G_0$ and the jump energy alone (cf. Lemma~\ref{l:lowbound} and 
Corollary~\ref{c:dlbjump} respectively)}.
For convenience, we introduce {for open sets $A\subset\Omega$} 
the deviation from minimality 
\[
\Psi_0(u,\cost,\beta,A):=G_0(u,\cost,\beta,A)-\Phi_0(u,\cost,\beta,A),
\] 
where 
\begin{equation}\label{e:phi0}
\Phi_0(u,\cost,\beta,A):=\inf\{G_0(v,\cost,\beta,A):\ v\in GSBD(\Omega), \ \{v\ne u\}\subset\hskip-0.125cm\subset A\}.
\end{equation}
The functions $f_\parametermu$ with $\parametermu>0$ and $\fz$ are both convex and with $p$-growth.
{The $p$-homogeneous function $\fz$ captures the asymptotic behavior of $f_\parametermu$ at infinity,}
\begin{equation}
 \fz(\xi)=\lim_{t\to\infty} \frac{{f_\parametermu}(t\xi)}{t^p}=\frac1p (\C\xi\cdot \xi)^{p/2}\,.
\end{equation}

{Before proceeding with the proofs we state an auxiliary result that will be repeatedly used 
in what follows (see \cite[Lemma~4.3]{ContiFocardiIurlano2016} for the elementary proof).
\begin{lemma}\label{affine}
Let $\omega\subseteq B_r(y)$ satisfy
\[
\calL^n(\omega)\leq\frac14 \calL^n(B_r(y)),
\]
and let $\varphi:\R^n\to\R^n$ be an affine function. Then
\[
\calL^n(B_r(y))\|\varphi\|_{L^\infty(B_r(y),\Rn)}\leq \bar{c}\|\varphi\|_{L^1(B_r(y)\setminus\omega,\Rn)},
\]
where the constant $\bar{c}$ depends only on the dimension $n$.
\end{lemma}

We investigate first the compactness properties of sequences having vanishing jump energy.}
\def\fraction{\sfrac 12}
\begin{proposition}\label{propcomp}
Let $n=2$, $p\in(1,\infty)$, 
 $B_\rho\subset\R^2$ a ball, $u_h\in {SBD^p}(B_\rho)$ and
 \begin{equation}\label{e:uhhp}
 \sup_h \int_{B_\rho} \fz(e(u_h))dx<\infty, \qquad \calH^1(J_{u_h})\to0.
 \end{equation}
{Then there are a function $u\in W^{1,p}(B_\rho;\R^2)$, a subsequence $h_j$, a sequence of affine functions $a_j:\R^2\to\R^2$ with $e(a_j)=0$, a sequence $z_j\in SBD^p(B_\rho)$ with 
\begin{enumerate}
	\item $\{z_j\ne u_{h_j}\}\subset\hskip-0.125cm\subset B_\rho$ and $\calL^2(\{z_j\ne u_{h_j}\})\to0$;
	\item $|z_j-a_j|\leq |u_{h_j}-a_j|$ and $|e(z_j)|\leq |e(u_{h_j})|$ 
	$\calL^2$-a.e. {on $B_\rho$};
	\item $\calH^1(J_{z_j}\setminus B_{\rho'})\leq c \calH^1(J_{u_{h_j}}\setminus B_{\rho''})$ for every $\rho''<\rho'<\rho$ and $j$ large, where $c$ is a universal constant;
	\item $z_j-a_j\to u$ $L^p_{\mathrm{loc}}(B_\rho;\R^2)$.
\end{enumerate}
	Moreover $u_{h_j}-a_j\to u$ $\calL^2$-a.e. {on $B_\rho$} and}
 \begin{equation}\label{e:tesiuh}
  \int_{B_\rho} \fz(e(u))dx \le {\liminf_{h\to\infty} \int_{B_\rho} \fz(e(u_{h}))dx\,.}
 \end{equation}
 \end{proposition}
\begin{proof}
Up to the extraction of a subsequence, we may assume that the inferior limit in {\eqref{e:tesiuh}}
is actually a limit.

For {each $h\in\N$ and for} any $s\in[\fraction,1)$ let $w_h^{(s)}\in SBD^p(B_\rho)$ and 
$\calF^s_h$ be the 
{function} and the family of balls obtained {by} Proposition~\ref{prop:ricopr} 
{applied to $u_h$}. By \eqref{e:volume} and Korn's inequality we can choose affine 
functions $a_h^{(s)}:\R^2\to\R^2$ such that $e(a_h^{(s)})=0$,
 \begin{equation*}
 \|Dw_h^{(s)}-Da_h^{(s)}\|_{L^p(B_{s\rho};\R^{2{\times}2})}\le c 
 \|e(w_h^{(s)})\|_{L^p(B_{s\rho};\R^{2{\times}2})}\le c 
 \|e(u_h)\|_{L^p(B_{\rho};\R^{2{\times}2})}
 \end{equation*}
 and $\|w_h^{(s)}-a_h^{(s)}\|_{{L^p}(B_{s\rho};\R^2)}\le  c {\rho}
 \|e(u_h)\|_{L^p(B_{\rho};\R^{2{\times}2})}$.
Now notice that 
{for $h$ large $\calL^2(B_{{\sfrac\rho 2}}\cap\{w_h^{(s)}={w^{(\fraction)}_h}=u_h\})
\geq\frac 14\calL^2(B_\rho)$ in view of item (vi) in Proposition~\ref{prop:ricopr}
and since $\calH^1(J_{u_h})\to0$ as $h\uparrow\infty$ (cf. \eqref{e:uhhp}). 
Thus, Lemma~\ref{affine}} and the triangular inequality imply for $h$ large
\begin{align}\label{affineauahfr}
 & \|Da_h^{(s)}-Da_h^{(\fraction)} \|_{L^p(B_{s\rho};\R^{2{\times}2})} \nonumber\\
 & \le c \|Da_h^{(s)}-Da_h^{(\fraction)}\|_{L^p(B_{{\sfrac\rho 2}}\cap
 \{w_h^{(s)}={w^{(\fraction)}_h}=u_h\};\R^{2{\times}2})} \le c 
 \|e(u_h)\|_{L^p(B_{\rho};\R^{2{\times}2})}
 \end{align}
 and similarly $\|a_h^{(\fraction)}-a_h^{(s)}\|_{{L^p}(B_{s\rho};\R^2)}\le  c {\rho}
 \|e(u_h)\|_{L^p(B_{\rho};\R^{2{\times}2})}$.
 
It follows that the sequence $w_h^{(s)}-a_h^{(\fraction)}$ is bounded in $W^{1,p}(B_{s\rho};\R^2)$ 
and therefore has a subsequence {(depending on $s$ {and not relabeled})} which converges to some 
$w^{(s)}$ weakly in $W^{1,p}(B_{s\rho};\R^2)$, strongly in $L^{q}(B_{s\rho};\R^2)$ for all 
$q\in[1,p^\ast)$ and pointwise $\calL^2$-a.e. on $B_{s\rho}$.
{Note that $\calL^2(\cup_{\calF^s_h}B)\leq \frac c\eta\rho\calH^1(J_{u_h})$ 
for all $s\in[\fraction,1)$ by item (vi) in Proposition~\ref{prop:ricopr}. Therefore, 
by \eqref{e:uhhp} we conclude that $w^{(s)}=w^{(t)}$ $\calL^2$-a.e. on $B_{s\rho}$ if 
$\fraction\leq s\leq t<1$. Thus, we may define a limit function $u$ on $B_\rho$ such 
that $u=w^{(s)}$ $\calL^2$-a.e. on $B_{s\rho}$ for all $s\in[\fraction,1)$.
In particular, $u\in W^{1,p}_{\loc}(B_\rho;\R^2)$.}

Let $B'\in \calF^s_h$. By the trace theorem,   
\begin{multline*}
\|w_h^{(s)}-a_h^{(\fraction)}\|_{L^1(\partial B';\R^2)} \\\le 
\frac{c}{r_{B'}} \|w_h^{(s)}-a_h^{(\fraction)}\|_{L^1(B';\R^2)} +
c\|Dw_h^{(s)}-Da_h^{(\fraction)}\|_{L^1( B';\R^{2{\times}2})} .
\end{multline*}
Using Korn's inequality, Poincar\'e's inequality,
	and \eqref{e:volume}, 
	we obtain that for each $B'$ there is an affine function $a_{B'}$ such that
	\begin{equation*}
	\frac{1} {r_{B'}} \|w_h^{(s)}-a_{B'}\|_{L^p(B';\R^2)} +
	\|Dw_h^{(s)}-Da_{B'}\|_{L^p(B';\R^{2{\times}2})} \le 
	c \|e(u_h)\|_{L^p(B';\R^{2{\times}2})} \,.
	\end{equation*}
	Since the center of $B'$ is contained in $\overline B_{s\rho}$, we have
	$\calL^2(B'\cap B_{s\rho}) \ge c \calL^2(B')$, and therefore, treating the affine function as in \eqref{affineauahfr},
	\begin{multline*}
	\|w_h^{(s)}-a_h^{(\fraction)}\|_{L^p(B';\R^2)}
		\le 
	\|w_h^{(s)}-a_{B'}\|_{L^p(B';\R^2)} 
	+c \|a_{B'}-a_h^{(\fraction)}\|_{L^p(B'\cap B_{s\rho};\R^2)}\\
	\le 
	c\|w_h^{(s)}-a_{B'}\|_{L^p(B';\R^2)} 
	+ c\|w_h^{(s)}-a_h^{(\fraction)}\|_{L^p(B'\cap B_{s\rho};\R^2)}.
	\end{multline*}
	The same holds for $Dw_h^{(s)}-Da_h^{(\fraction)}$. We conclude
	\begin{multline*}
	\sum_{B'\in\calF^s_h} \frac{1} {r_{B'}} \|w_h^{(s)}-a_h^{(\fraction)}\|_{L^p(B';\R^2)} +
	\|Dw_h^{(s)}-Da_h^{(\fraction)}\|_{L^p(B';\R^{2{\times}2})}\\
	\le c \|e(u_h)\|_{L^p(B_\rho;\R^{2{\times}2})}
	\end{multline*}
	and therefore, since $r_{B'}\le \rho$,
	\begin{align*}
	\sum_{B'\in\calF^s_h} \|w_h^{(s)}-a_h^{(\fraction)}\|_{L^1(\partial B';\R^2)}  &
	\le c \rho^{2-\sfrac2p}\|e(u_h)\|_{L^p(B_\rho;\R^{2{\times}2})}.
	\end{align*}
We define $z_h^{(s)}:=w_h^{(s)} + (a_h^{(\fraction)}-w_h^{(s)}) \chi_{\cup_{\calF^s_h}B}$.
The previous estimates show that {$z_h^{(s)}-a_h^{(\fraction)}\in SBD^p(B_{\rho})$.}
Moreover, $\calH^1(J_{z_h^{(s)}}{\cap B_{s\rho}})\leq\sum_{B\in \calF^s_h} \calH^1(\partial B)$,
{$e(z_h^{(s)})=e(u_h)\chi_{B_{\rho}\setminus \cup_{\calF^s_h}B}$} $\calL^2$-a.e., 
so that the sequence $z_h^{(s)}-a_h^{(\fraction)}$ is bounded in $SBD^p(B_{s\rho})$. 
In addition, for all $q\in(p,p^\ast)$ it holds
\begin{multline*}
\|z_h^{(s)}-w_h^{(s)}\|_{L^p(B_{s\rho};\R^2)}= \|a_h^{(\fraction)}-w_h^{(s)}\|_{L^p({B_{s\rho}\cap}\cup_{\calF^s_h}B;\R^2)}
\\
\leq \big(\calL^2(\cup_{\calF^s_h}B)\big)^{\sfrac1p-\sfrac 1q}\|a_h^{(\fraction)}-w_h^{(s)}\|_{L^q({B_{s\rho}};\R^2)}.
\end{multline*}
Hence, $z_h^{(s)}-w_h^{(s)}\to 0$ in $L^p(B_{s\rho};\R^2)$ since $\calL^2(\cup_{\calF^s_h}B)\leq c(\calH^1(J_{u_h}))^2$
for some universal constant $c>0$.
Therefore, $z_h^{(s)}-a_h^{(\fraction)}$ has a subsequence 
{(depending on $s$)} converging $L^p(B_{s\rho};\R^2)$ and $\calL^2$-a.e. on 
$B_{s\rho}$ to {$u$} . 

{From 
	{$e(z_h^{(s)})=e(w_h^s)\chi_{B_{\rho}\setminus \cup\calF^s_h}$}
	one sees that 
	$e(z_h^{(s)})$ 
	converges weakly in $L^p$ to {$e(u)$} .}
{{Hence, recalling that we have assumed 
	the inferior limit in {\eqref{e:tesiuh}} to be a limit,} by convexity 
	{and positivity} of $\fz$ we obtain {for all $s\in[\fraction,1)$}
	\begin{equation}\label{e:tesiuhs}
	\int_{B_{s\rho}} \fz({e(u)}) dx \le \liminf_{h\to\infty} \int_{B_{s\rho}} \fz(e(z^s_h))dx
	\le \liminf_{h\to\infty} \int_{B_{s\rho}} \fz(e(u_h))dx,
	\end{equation}
	where in the second step we used $e(z_h^s)=e(u_h)\chi_{B_{s\rho}\setminus \cup \calF_h^s}$.
}
{From this we conclude that} $u\in W^{1,p}(B_\rho;\R^2)$,  
{and moreover the lower semicontinuity estimate in} \eqref{e:tesiuh} follows at once 
{being $f_0$ nonnegative}. 

{Eventually,} take a sequence $s_j\uparrow 1$,  
for every $j\in\N$ let $h_j$ be such that 
\[
\|z^{(s_j)}_{h_j}-a_{h_j}^{\sfrac 12}-{u}\|_{L^p(B_{s_j\,\rho};\R^2)}\leq \sfrac 1j,
\]
set $z_j:=z^{(s_j)}_{h_j}$ {and $a_j:=a_{h_j}^{{(\sfrac 12)}}$}, then {properties (i)-(iv) follow by construction}. 
 
{Finally, the sequence $u_{h_j}-a_j$ converges in measure to $u$ by item (iii) in Proposition~\ref{prop:ricopr}
and since $\calL^2(\cup_{\calF^{s_j}_{h_j}}B)$ is infinitesimal as already noticed.}
\end{proof}

\begin{remark}
 The result above extends to sequences $u_h\in GSBD^p(B_\rho)\cap L^p(B_\rho;\R^2)$ 
 by using the approximation argument that will be employed in Proposition~\ref{propconvenerg}
 below.
\end{remark}

{We investigate next the asymptotics of sequences with vanishing jump energy.}
\begin{proposition}\label{propconvenerg}
Let $n=2$, $p\in(1,\infty)$.
Let $B_r$ be a ball, $u_h\in GSBD^p(B_r)$ and {$\cost_h\in[0,\infty)$,} $\beta_h\in(0,\infty)$ be two sequences 
with $\cost_h\to 0$ as $h\to\infty$, and such that
 \begin{equation*}
  \sup_h G_0(u_h,\cost_h,\beta_h, B_r)<\infty,\text{ and  } \lim_{h\to\infty} \Psi_0(u_h,\cost_h,\beta_h, B_r)=
  \lim_{h\to\infty} \calH^1(J_{u_h})=0\,.
 \end{equation*}
Then there exists $u\in W^{1,p}(B_r;\R^2)$, $\overline{a}:\R^2\to\R^2$ affine with $e(\overline{a})=0$
and a subsequence $h_j$
such that
\begin{enumerate}
 
 \item\label{propconvenerg-itrho}  for all $\rho\in (0,r)$
 \[ 
 \lim_{j\to\infty} G_0(u_{h_j},\cost_{h_j},\beta_{h_j},B_\rho)=\int_{B_\rho} \fz(e(u)) dx+
 \int_{B_\rho}|\overline{a}|^pdx;
 \]
 
 \item\label{propconvenerg-itmin} for all $v\in u+W^{1,p}_0(B_r;\R^2)$ 
\[
\int_{B_r} \fz(e(u))dx\le \int_{B_r} \fz(e(v))dx;
\]

\item \label{propconvenerg-remark}
$u_{h_j}-a_{j}\to u$ pointwise $\calL^2$-a.e. on $B_r$ for some affine functions $a_j$, 
{$e(u_{h_j})\to e(u)$ in $L^p(B_\rho;\R^{2\times 2})$, ${\beta_{h_j}\calH^1(J_{u_{h_j}}\cap B_\rho)}\to 0$, 
and ${\cost_{h_j}^{\sfrac 1p}u_{h_j}}\to\overline{a}$ in $L^p(B_\rho;\R^2)$ for all $\rho\in(0,r)$.}
\end{enumerate}
\end{proposition}
\begin{proof}
\def\w{v}
\def\rhoout{\rho}
\def\rhomid{{\rho'}}
\def\rhoin{{\rho''}}
\def\rhobar{{\overline{\rho}}}
\def\s{s}
\def\rhoinin{{\rho'''}}

{Theorem}~\ref{flaviana}
provides $\w_h\in SBV^p\cap L^\infty(B_r;\R^2)$, for every $h\in\N$, 
such that 
\begin{multline}\label{e:density}
\|e(u_h)-e(\w_h)\|_{L^p(B_r;\R^{2\times 2})}+{|\calH^{1}(J_{u_h})-\calH^{1}(J_{v_h})|}\\
+\|u_h-\w_h\|_{L^p(B_r;\R^2)}\leq (h+\beta_h^2)^{-1}.
\end{multline}
In particular, {for all $\rho\in(0,r]$ 
\begin{equation}\label{e:G0uhG0vh}
\limsup_{h\to\infty} G_0(u_h,\cost_h,\beta_h,B_\rho)=
\limsup_{h\to\infty} G_0(\w_h,\cost_h,\beta_h,B_\rho),
\end{equation}}
and  
\begin{equation*}
\lim_{h\to\infty} \calH^1(J_{\w_h})=0\,.
\end{equation*}
Hence, $(\w_h)_{h\in\N}$ satisfies \eqref{e:uhhp} in Proposition~\ref{propcomp}.
Let {$a_{h_j}$} and $u$ be the functions obtained by Proposition~\ref{propcomp},
then ${v_{h_j}-a_{h_j}}\to u$ pointwise {$\calL^2$-a.e. on $B_r$.
Recall that Proposition~\ref{propcomp} {and \eqref{e:density}}  imply that
\begin{equation}\label{eqfualpharho}
\int_{B_\rho} \fz(e(u)) dx \le \liminf_{h\to\infty} \int_{B_\rho} \fz(e(u_h))dx.
\end{equation}
Additionally, {up to extracting a further subsequence} we may assume that {$u_{h_j}-a_{h_j}\to u$} pointwise $\calL^2$-a.e. on $B_r$ by \eqref{e:density}.
Here and henceforth we denote $h_j$ by $h$ for simplicity.
}

Since $\s\mapsto G_0(u_h,\kappa_h,\beta_h,B_\s)$
is nondecreasing and uniformly bounded, 
by Helly's theorem we can extract a subsequence, not relabeled for convenience, 
such that the pointwise limit
\begin{equation}\label{e:lambda}
\lim_{h\to\infty}G_0(u_h,\kappa_h,\beta_h,B_\s)=:\Lambda(\s)
\end{equation}
exists finite for $\calL^1$-a.e. $\s\in(0,r)$, and $\Lambda$ is a  nondecreasing function. 
{Define $I\subseteq(0,r)$ to be the set of radii where \eqref{e:lambda} holds true.}

Being $(\cost_h^{\sfrac 1p}u_h)_h$ bounded in $L^p(B_r;\R^2)$, 
it has a subsequence (not relabeled) converging to some 
${\overline a}\in L^p(B_r;\R^2)$ weakly in $L^p(B_r;\R^2)$.
At the same time  $\kappa_h^{1/p}(u_h-a_h)\to0$ pointwise $\calL^2$-a.e. in $B_r$, {as 
$\kappa_h\downarrow 0$ as $h\to\infty$,}
therefore {$(\kappa_h^{1/p}a_h)_h$ is bounded in $L^p(B_r;\R^2)$ by 
{Lemma~\ref{affine}}. 
Hence, by the Urysohn property, by the weak $L^p$-convergence of 
$(\kappa_h^{1/p}u_h)_h$} and by the {equiintegrability of $(\cost_h^{\sfrac 1p}(u_h-a_h))_h$}
we obtain that in turn $(\kappa_h^{1/p}a_h)_h$ converges weakly to $\bar a$ in  $L^p(B_r;\R^2)$. 
Since $\kappa_h^{1/p}a_h$ are affine functions, and the space of affine functions is finite dimensional, 
convergence is actually strong, and $\overline a$ is affine on $B_r$, with $e(\bar a)=0$. 

Fixed {$\rhoout\in I$} a continuity point of $\Lambda$ satisfying \eqref{e:lambda}
we apply Proposition~\ref{propcomp} again to $B_\rhoout$ and obtain a subsequence of $h$ not relabeled, 
a sequence $(z_h^{(\rho)})_h\in SBD^p({B_\rhoout})$, and a sequence $a_h^{(\rho)}:\R^2\to\R^2$ of affine 
functions with $e(a_h^{(\rho)})=0$, such that $\w_h-a_h^{(\rho)}\to u^{(\rho)}$ $\calL^2$-a.e. on $B_\rhoout$, 
$z_h^{(\rho)}-a_h^{(\rho)}\to {u^{(\rho)}}$ in $L^p_\loc(B_\rhoout;\R^2)$ and 
$\{z_h^{(\rho)}\neq \w_h\}\subset\hskip-0.125cm\subset B_\rhoout$, for some $u^{(\rho)}\in W^{1,p}(B_\rho;\R^2)$. 
{Thus, we may consider $z_h^{(\rho)}$ as a function in $SBD^p(B_r)$ by extending it equal to $v_h$ 
on $B_r\setminus B_\rho$.}

{Next note that $z_h^{(\rho)}-a_h\to u$ in  $L^p_\loc(B_\rho;\R^2)$, where $a_h$ and $u$ are the 
globally chosen functions introduced above. 
This claim easily follows from the convergences $\w_h-a_h\to u$ $\calL^2$-a.e. 
on $B_r$ and $\w_h-a_h^{(\rho)}\to u^{(\rho)}$ $\calL^2$-a.e. on $B_\rhoout$. Indeed, from these 
we deduce that $a_h^{(\rho)}-a_h\to u-u^{(\rho)}$ in $L^p(B_\rho;\R^2)$.
Hence, the claim follows at once by taking into account this and the convergence
$z_h^{(\rho)}-a_h^{(\rho)}\to {u^{(\rho)}}$ in $L^p_\loc(B_\rhoout;\R^2)$.}

Let $v\in W^{1,p}(B_r;\R^2)$ be such that $\{u\neq v\}\subset\hskip-0.125cm\subset B_\rho$ and let 
$0<\rhoinin<\rhoin<\rhomid<\rhoout<\rhobar<r$, with $\rhoinin,\,\rhobar\in I$ and 
{assume in addition that} $\{u\neq v\}\subseteq B_\rhoin$.

Let $\zeta\in C^\infty_c(B_{\rhomid};[0,1])$, $\varphi\in C^\infty_c(B_{\rhobar};[0,1])$ be cut-off
functions such that $\zeta=1$ on $B_\rhoin$, $\varphi=1$ on $B_\rhoout$, and
$\|\nabla\zeta\|_{L^\infty(B_\rhomid;\R^2)}\leq 2(\rhomid-\rhoin)^{-1}$, 
$\|\nabla\varphi\|_{L^\infty(B_{\rhobar};\R^2)}\leq 2(\rhobar-\rhoout)^{-1}$. Define
\[
\overline{u}_h:=\zeta (v+a_h) + (1-\zeta) \big(\varphi\, {z_h^{(\rho)}}+(1-\varphi)u_h\big)
\] 
and note that
\[
\overline{u}_h= \begin{cases}
\zeta (v+a_h) + (1-\zeta) {z_h^{(\rho)}} & \text{on $B_\rhomid$} \\
\varphi\, {z_h^{(\rho)}}+(1-\varphi)u_h   & \text{on $B_r\setminus B_\rhomid$.}
\end{cases}
\]
Since $\{\overline{u}_h \ne u_h\}\subset\hskip-0.125cm\subset B_\rhobar$, by the very definition of $\Psi_0$ we have
\begin{equation}\label{e:uhquasimin}
G_0(u_h,\cost_h,\beta_h,B_\rhobar)\le G_0(\overline{u}_h,\cost_h,\beta_h,B_\rhobar)+\Psi_0(u_h,\cost_h,\beta_h,B_r).
\end{equation}
We estimate separately the contributions on $B_{\rhomid}$ and $B_\rhobar\setminus B_{\rhomid}$ 
for the first summand on the right hand side above as follows. First, for some $c=c(p)>0$ we have
\begin{align*}
G_0 & \,(\overline{u}_h,\cost_h,\beta_h,B_{\rhomid}) \leq  G_0(v+a_h,\cost_h,\beta_h,B_{\rhoin})+
c\,G_0(v+a_h,\cost_h,\beta_h,B_{\rhomid}\setminus B_{\rhoin})\notag\\
& +c\,G_0({z_h^{(\rho)}},\cost_h,\beta_h,B_{\rhomid}\setminus B_{\rhoin})+\frac c{(\rhomid-\rhoin)^p}
\int_{B_{\rhomid}\setminus B_{\rhoin}}|v+a_h-{z_h^{(\rho)}}|^pdx\notag\\
& = \int_{B_\rhoin}\fz(e(v))dx+ \cost_h\int_{B_\rhoin}|v+a_h|^pdx \notag\\
& +c\int_{B_{\rhomid}\setminus B_{\rhoin}}\fz(e(v))dx+c\, \cost_h\int_{B_{\rhomid}\setminus B_{\rhoin}}|v+a_h|^pdx\notag\\
& +c\,G_0({z_h^{(\rho)}},\cost_h,\beta_h,B_{\rhomid}\setminus B_{\rhoin})+\frac c{(\rhomid-\rhoin)^p}
\int_{B_{\rhomid}\setminus B_{\rhoin}}|v+a_h-{z_h^{(\rho)}}|^pdx.
\end{align*}
Moreover, since $\{{z_h^{(\rho)}}\neq \w_h\}\subset\hskip-0.125cm\subset B_\rhoout$, {and $\overline{u}_h=z_h^{(\rho)}$ on 
$B_\rho\setminus B_{\rhomid}$} we have
\begin{align*}
G_0(\overline{u}_h,\cost_h,\beta_h,{B_\rhobar\setminus B_{\rhomid}}) 
\leq  & c\,G_0({z_h^{(\rho)}},\cost_h,\beta_h,
{B_{\rhobar}\setminus B_{\rhomid}})
+c\, G_0(u_h,\cost_h,\beta_h,B_\rhobar\setminus B_{\rhoout})\\
& +\frac c{(\rhobar-\rhoout)^p}\int_{B_{\rhobar}\setminus B_{\rhoout}}|\w_h-u_h|^pdx.
\end{align*}
Therefore, since  $u=v$ on $B_\rhoout\setminus B_\rhoin$ 
we deduce that
\begin{align}\label{e:Gzh3}
G_0 & \,(\overline{u}_h,\cost_h,\beta_h,B_{\rhobar}) \leq 
\int_{B_\rhoin}\fz(e(v))dx+\cost_h\int_{B_\rhoin}|v+a_h|^pdx\notag\\
+ & c\int_{B_{\rhomid}\setminus B_{\rhoin}}\fz(e(v))dx 
+c\, \cost_h\int_{B_{\rhomid}\setminus B_{\rhoin}}|u+a_h|^pdx \notag\\ 
+ & c\,G_0({z_h^{(\rho)}},\cost_h,\beta_h,B_{\rhobar}\setminus B_{\rhoin})
+c\, G_0(u_h,\cost_h,\beta_h,B_\rhobar\setminus B_{\rhoout})\notag \\
+ & \frac c{(\rhomid-\rhoin)^p}\int_{B_{\rhomid}\setminus B_{\rhoin}}|u+a_h-{z_h^{(\rho)}}|^pdx
+\frac c{(\rhobar-\rhoout)^p}\int_{B_{\rhobar}\setminus B_{\rhoout}}|\w_h-u_h|^pdx.
\end{align}
Note that  {by (ii)-(iii) in} Proposition~\ref{propcomp} 
we have, {for $h$ sufficiently large,}
\[
G_0({z_h^{(\rho)}}, \cost_h,\beta_h,B_{\rhobar}\setminus B_{\rhoin}) \le c\,
G_0(v_h,\cost_h,\beta_h,B_{\rhobar}\setminus B_{\rhoinin}).
\]
therefore 
\eqref{e:G0uhG0vh}, \eqref{e:lambda} and the choices of the radii $\rhoinin,\rhobar\in I$ yield 
\begin{equation*}
\limsup_{h\to\infty} \big(
G_0({z_h^{(\rho)}},\cost_h,\beta_h,B_{\rhobar}\setminus B_{\rhoin})
+G_0(u_h,\cost_h,\beta_h,B_\rhobar\setminus B_{\rhoout})\big)\leq c(\Lambda(\rhobar)-\Lambda(\rhoinin)).
\end{equation*}
Moreover, recalling the convergences $u_h-\w_h\to 0$ in $L^p(B_r;\R^2)$, ${z_h^{(\rho)}}-a_h\to u$ 
$L^p(B_\rhomid;\R^2)$, $\cost_h^{\sfrac 1p}a_h\to\overline{a}$ in $L^p(B_r;\R^2)$ and 
$\kappa_h\to 0$ as $h\to\infty$, we infer
\begin{align*}
\lim_{h\to\infty} & \Big(\cost_h\int_{B_{\rhomid}\setminus B_{\rhoin}}|u+a_h|^pdx
+\frac 1{(\rhomid-\rhoin)^p}\int_{B_{\rhomid}\setminus B_{\rhoin}}|u+a_h-{z_h^{(\rho)}}|^pdx
\\
& +\frac 1{(\rhobar-\rhoout)^p}\int_{B_{\rhobar}\setminus B_{\rhoout}}|\w_h-u_h|^pdx\Big) 
{=\int_{B_{\rhomid}\setminus B_{\rhoin}}|\overline{a}|^pdx,}
\end{align*}
Hence, by taking the superior limit as $h\to\infty$ in \eqref{e:uhquasimin}, in view of \eqref{e:Gzh3} 
and the last two inequalities we get 
\begin{align*}
&\Lambda(\rhobar)
\leq \int_{B_\rhoin}\fz(e(v))dx+\int_{B_\rhoin}|\overline{a}|^pdx
+c\int_{B_{\rhomid}\setminus B_{\rhoin}}\fz(e(v))dx\\& +c\, \int_{B_{\rhomid}\setminus B_{\rhoin}}|\overline{a}|^pdx 
+ c(\Lambda(\rhobar)-\Lambda(\rhoinin)).
\end{align*}
On the other hand, the weak convergence of $(\cost_h^{\sfrac 1p}u_h)_h$ to $\overline{a}$ {and \eqref{eqfualpharho}}
yield
\begin{equation}\label{eqf0alamb}
\int_{B_\rhoout} \fz(e(u))dx+\int_{B_\rhoout}|\overline{a}|^pdx \leq
\liminf_{h\to\infty} \int_{B_\rhoout} \big(\fz(e(u_h)) + \cost_h|u_h|^p\big)\,dx \leq \Lambda(\rhoout). 
\end{equation}
Therefore, from the last two inequalities we conclude as $\rhoinin,\rhobar\to\rhoout$
\begin{align}\label{e:propconvenerg-itrho}
\int_{B_\rhoout} \fz(e(u))dx\,+& \int_{B_\rhoout}|\overline{a}|^pdx 
\notag\\
&\leq \Lambda(\rhoout) \leq\int_{B_\rhoout} \fz(e(v))dx+\int_{B_\rhoout}|\overline{a}|^pdx,
\end{align}
and thus in particular
\begin{equation}\label{e:umin}
\int_{B_\rhoout} \fz(e(u))dx\leq  \int_{B_\rhoout} \fz(e(v))dx
\end{equation}
for all $v\in W^{1,p}(B_r;\R^2)$ such that $\{u\neq v\}\subset\hskip-0.125cm\subset B_\rhoout$ and for $\calL^1$ a.e. 
$\rhoout\in(0,r)$. 
Clearly, a simple approximation argument yields that the inequality \eqref{e:umin} holds for all 
$v\in u+W^{1,p}_0(B_r;\R^2)$, i.e. item \ref{propconvenerg-itmin} is established.

Finally, setting $v=u$ in \eqref{e:propconvenerg-itrho}, we deduce that for $\calL^1$ a.e. $\rhoout\in(0,r)$
\[
\int_{B_\rhoout} \fz(e(u))+ |\overline{a}|^pdx=\Lambda(\rhoout).
\]
Being the left-hand side there continuous as a function of $\rhoout$, ${\Lambda}$ turns out to be continuous 
as well, and recalling its very definition and the monotonicity of the integral we conclude that convergence 
in \eqref{e:lambda}
holds for all $\rhoout\in(0,r)$, i.e. item \ref{propconvenerg-itrho} is established as well. 
Furthermore, from this and \eqref{e:propconvenerg-itrho}
above one deduces {that equality holds in \eqref{eqf0alamb}, and therefore that the convergence of $e(u_h)$ 
and $\kappa_h^{1/p} u_h$ is strong, which concludes the proof of}
\ref{propconvenerg-remark}.
\end{proof}

{We are now ready to prove a fundamental decay property of $G_0$ by following
the ideas in \cite[Lemma 3.9]{CarrieroLeaci91}. 
Nevertheless, we note explicitly that contrary to \cite[Lemma 3.9]{CarrieroLeaci91} the lack 
of truncation arguments forces to take also into account the fidelity term in the decay process, 
since a priori we have no $L^\infty$ bound on local minimizers.}
{As part of the argument extends directly to higher dimension, we give a proof of the density lower bound that depends only on the decay property. However,  the decay property has been proven using  the regularity of Sobolev minimizers as well as Propositions \ref{propcomp} and \ref{propconvenerg}, which have only been established in dimension $n=2$.}

{
\begin{definition}\label{d:decay}
Let $n\ge 2$, $p\in(1,\infty)$, $\cost\ge 0$, $\beta>0$.
 We say that the decay property holds for the functional $G_0$ in dimension $n$ if
the following is true. For any
  $\gamma\in(0,1)$ there is 
 $\tau_\gamma>0$  such that  for all 
$\tau\in(0,\tau_\gamma]$ there exist $\eps\in(0,1)$, $\vartheta\in(0,1)$, and $R>0$, such that if 
$u\in GSBD^p(\Omega)$ satisfies
\begin{equation*}
{\calH^{n-1}(J_u\cap {B}_\rho(x))}\leq\eps\rho^{n-1} \text{ and }
G_0(u,\cost,\beta,B_\rho(x))\leq(1+\vartheta) \Phi_0(u,\cost,\beta,B_\rho(x))
\end{equation*}
for some
${B}_\rho(x)\subset\hskip-0.125cm\subset\Omega$ with $0<\rho<R$, then
\[
 G_0(u,\cost,\beta, B_{\tau\rho}(x)) \le \tau^{n-\gamma} G_0(u,\cost,\beta, B_\rho(x)).
\]
 \end{definition}
}

{
\begin{lemma}[Decay]\label{lemmadecay}
The decay property holds in dimension $n=2$ for any
 $p\in(1,\infty)$, $\cost\ge 0$, $\beta>0$.
\end{lemma}}
In what follows $c_\gamma$ denotes the constant in Proposition~\ref{p:decayVmu} having chosen 
$\gamma>0$, and $\bar{c}$ that of Lemma~\ref{affine}. 
\begin{proof}[{Proof of Lemma \ref{lemmadecay}}]
{
Let $\tau_\gamma>0$ be such that 
$\max\{c_{\sfrac\gamma2}\tau_\gamma^{\sfrac\gamma 2},{\bar{c}^p\,\tau_\gamma^\gamma},\tau_\gamma\}=\sfrac 12$.}

By contradiction suppose the statement false. Then there would be {$\tau\in(0,\tau_\gamma]$ and}
three sequences $\eps_h\to0$, $\vartheta_h\to0$, $\rho_h\to0$, a sequence $u_h\in {GSBD^p(\Omega)}$, 
and a sequence of balls ${B_{\rho_h}(x_h)\subset\hskip-0.125cm\subset\Omega}$ such that 
\begin{align*}
&{\calH^{1}({J_{u_h}}\cap B_{\rho_h}(x_h))}= \eps_h\,\rho_h,\\
&G_0(u_h,\cost, \beta, B_{\rho_h}(x_h)) ={(1+\vartheta_h)} \Phi_0(u_h,\cost,\beta,B_{\rho_h}(x_h)),
 \end{align*}
 with
 \begin{equation*}
 G_0(u_h,\cost,\beta, B_{\tau\rho_h}(x_h)) > \tau^{2-\gamma} G_0(u_h,\cost,\beta, B_{\rho_h}(x_h)).  
 \end{equation*}
We define
\begin{equation*}
 \sigma_h:= \frac{\rho_h}{G_0(u_h,\cost,\beta, B_{\rho_h}(x_h))} \quad\text{ and }\quad 
 v_h(y):=\frac{(\sigma_h\rho_h)^{1/p}}{\rho_h} u_h(x_h+\rho_h y)
\end{equation*}
so that $v_h\in {GSBD^p}(B_1)$ satisfies $\calH^{1}(J_{v_h})=\eps_h$, $G_0(v_h,\cost\rho_h^p,\beta\sigma_h, B_1)=1$, $\Psi_0(v_h,\kappa\rho_h^p,\beta\sigma_h, B_1)={\vartheta_h/(1+\vartheta_h)}$, and
\be{\label{e:contr}G_0(v_h,\cost\rho_h^p, \beta\sigma_h, B_\tau)>\tau^{2-\gamma}.}

{By Proposition~\ref{propconvenerg}} there exist a subsequence $h$ not relabeled, 
a function $v\in W^{1,p}(B_1;\R^{2})$, and affine functions $a_h$ such that $v_h-a_h\to v$
{$\calL^{2}$-a.e. on $B_1$}, and for some affine function $\overline{a}$ with $e(\overline{a})=0$
\begin{equation}\label{e:bounda}
\int_{B_{\rho }} \fz(e(v)) dx 
 +\int_{B_\rho}|\overline{a}|^pdx
= \lim_{h\to\infty}G_0(v_h,\cost\rho_h^p,\beta\sigma_h,B_\rho)\le 1
\end{equation}
for all $\rho\in (0,1)$, with $v$ a {minimizer} of $w\mapsto\int_{B_1} \fz(e(w))dx$
among all $w\in v+W^{1,p}_0(B_1;{\R^{2}})$. 

Hence, by Proposition~\ref{p:decayVmu}, applied with the exponent $\gamma/2$, 
by {Lemma~\ref{affine}} 
and \eqref{e:bounda}
\begin{align*}
\lim_{h\to\infty}& G_0(v_h,\cost\rho_h^p,\beta\sigma_h,B_\tau)= \int_{B_{\tau }} \fz(e(v)) dx
 +\int_{B_\tau}|\overline{a}|^pdx\\
 \le & c_{\sfrac\gamma2}  \tau^{2-\gamma/2}
 +\|\overline{a}\|_{L^\infty(B_\tau;\R^2)}^p{\mathcal{L}^2(B_1)\tau^{2}}
 \leq \big(c_{\sfrac\gamma2}  \tau^{\gamma/2}+{\bar{c}^p}\tau^{\gamma}\big)\tau^{2-\gamma}
 <\tau^{2-\gamma}, 
\end{align*}
where the last inequality follows by the definition of $\tau_\gamma$. This contradicts \eqref{e:contr}. 
\end{proof}
{
\begin{remark}
The conclusions of {Proposition}~\ref{p:decayVmu} actually hold without dimensional limitations (cf. \cite[Proposition~3.4]{ContiFocardiIurlanoRegularity}), but are clearly not enough to deduce in higher dimensions the decay property (cf. 
Definition~\ref{d:decay}).
\end{remark}
}

{We finally  establish the density lower bound for the homogeneous energy $G_0$ and for the jump term. 
The proof of the next result follows the lines of \cite[Lemma 4.3]{FonsecaFusco97}. 
As this argument does not depend  
on dimension except for the decay property we formulate it for general $n$. 
We denote by $J_u^\ast$ the set of points  $x\in {J_u}$ with density one, namely
\begin{equation}\label{e:densone}
J_u^\ast:=\left\{x\in J_u: 
\lim_{\rho\to 0}\frac{\calH^{n-1}(J_u\cap B_\rho(x))}{{\omega_{n-1}\rho^{n-1}}}=1\right\},
\end{equation}
where} {$\omega_{n-1}$ is the $(n-1)$-dimensional Lebesgue measure of 
the unit ball in $\R^{n-1}$.}
\begin{lemma}[Density lower bound for $G_0$]\label{l:lowbound}
{
Let $n\ge 2$, $p>1$, $\cost\ge0 $, $\beta>0$, $\mu\ge0$,
 $g\in L^\infty(\Omega;\R^n)$. 
Assume the decay property holds for $G_0$ in dimension $n$. }
If $u\in GSBD^p(\Omega)$ is a local minimizer of
$G(\cdot,\cost,\beta,\Omega)$ defined in \eqref{e:G}, then there exist $\vartheta_0$ 
and $R_0$, depending only on ${n}$, $p$, $\mathbb{C}$, $\cost$, $\beta$, $\parametermu$, and 
$\|g\|_{L^\infty(\Omega;\R^{n})}$, such that if $0<\rho<R_0$, {$x\in\Omega\cap \overline{J_u^*}$,} and 
$B_\rho(x)\subset\hskip-0.125cm\subset\Omega$, then 
\be{\label{e:lb}
{G_0(u,\cost,\beta,B_\rho(x))\geq \vartheta_0\rho^{n-1}.}}
\end{lemma}
\begin{proof}
 {Let us first assume that $x\in J_u^*$.}

\noindent\emph{Step 1.} 
{We choose $\gamma=1/4$ in the decay property (Definition~\ref{d:decay})}
and choose $\tau\in(0,2^{-16}\wedge\tau_{\sfrac 14})$,
{with $\tau_{\sfrac 14}$ as in Definition~\ref{d:decay}.}
Let
{$\eps:={\omega_{n-1}}\wedge\eps(\tau)$}, where $\eps(\tau)\in(0,1)$,
 $\vartheta{=\vartheta(\tau)}\in(0,1)$, 
and $R{=R(\tau)}>0$, are as in the decay property.

We {claim} that there exists a radius $R_1=R_1{({n,}\tau,\mu,p,\|g\|_{L^\infty(\Omega;\R^{n})})}>0$ 
such that if
\be{\label{e:s0}
G(u,\cost,\beta,B_\rho(x))< \beta\,\eps\rho^{n-1}
}
for some  $0<\rho<R_1$, 
then one of the following inequalities holds
\ba{
&G(u,\cost,\beta,B_{\tau\rho}(x))<\tau^{n-1}\rho^{n-\sfrac12},\label{e:s1}\\
&G(u,\cost,\beta,B_{\tau\rho}(x))<\tau^{n-\sfrac12}G(u,\cost,\beta,B_\rho(x)).\label{e:s2}
}
We distinguish two cases. If 
\be{\label{e:s3}
G(u,\cost,\beta,B_{\tau\rho}(x))<\rho^{n-\sfrac14},
}
then \eqref{e:s1} holds provided we choose $R_1\le\tau^{4(n-1)}$.

To deal with the remaining case we state two elementary inequalities: for any $\sigma>0$ 
there is ${k_\sigma}>1$ (implicitly depending also on $p$) such that

\begin{equation}\label{e:easy}
|z+\zeta|^p\leq (1+\sigma)|z|^p+k_{\sigma}|\zeta|^p\qquad
\text{for all $z,\,\zeta\in\R^{n}$}
\end{equation}
and
\begin{equation}\label{eqf0fmu}
 f_0(\xi)-\parametermu^{p/2}\le f_\parametermu(\xi)\le (1+\sigma) f_0(\xi)+k_\sigma\parametermu^{p/2}
 \qquad\text{ for all $\xi\in \R^{n\times n}$}.
\end{equation}
Using \eqref{eqf0fmu} and \eqref{e:easy} with $\sigma=1$, and the fact that 
$g\in L^\infty(\Omega;\R^{n})$, we get 
\begin{align*}
 G(u,\cost,\beta, B_{\tau\rho}(x)) \le 2 G_0(u,\cost,\beta, B_{\tau\rho}(x)) 
 +k_1(\parametermu^{p/2}+\|g\|^p_{L^\infty(\Omega;\R^2)})\rho{^n\omega_n}\,.
\end{align*}
Since  \eqref{e:s3} does not hold, choosing ${R_1\leq R}$ such that 
\begin{equation}\label{e:R11}
{8\omega_n k_1(\parametermu^{p/2}+\|g\|^p_{L^\infty(\Omega;\R^n)}) R_1^{1/4}\le 1}
\end{equation}
we obtain
\begin{equation}\label{e:treG0}
 G(u,\cost,\beta, B_{\tau\rho}(x)) \le 4 G_0(u,\cost,\beta, 
 B_{\tau\rho}(x)).
\end{equation}
Suppose now that
\begin{equation}\label{F0phi0}
G_0(u,\cost,\beta,B_{\rho}(x))\leq(1+\vartheta)\Phi_0(u,\cost,\beta,B_{\rho}(x)).
\end{equation}
Then, by \eqref{e:s0} and \eqref{F0phi0} the decay property,
 \eqref{eqf0fmu}, \eqref{e:treG0} and  $g\in L^\infty(\Omega;\R^{n})$ yield 
\begin{align*}
 G(u,\cost,\beta, B_{\tau\rho}(x)) & \le 4 \tau^{n-\sfrac14} G_0(u,\cost,\beta, B_{\rho}(x))\\
& \le 8\tau^{n-\sfrac14} G(u,\cost,\beta, B_\rho(x)) + 4{\omega_n}
(\parametermu^{\sfrac p2}+k_1\|g\|^p_{L^\infty(\Omega;\R^{n})})\rho^n.
\end{align*}
Since \eqref{e:s3} is not satisfied, as above we can absorb the last term in the left-hand side by taking into account the condition
in \eqref{e:R11} to obtain
\[
 G(u,\cost,\beta, B_{\tau\rho}(x))  \le 16\tau^{n-\sfrac14}G(u,\cost,\beta, B_\rho(x)) .
\]
The proof of \eqref{e:s2} is concluded since $16\tau^{\sfrac 14}<1$.

Hence, we are left with proving \eqref{F0phi0} assuming that \eqref{e:s3} is violated.
To this aim we first fix $\sigma={\sigma(\tau)}\in(0,\sfrac 12)$ such that 
\be{\label{e:sigma}
(1+2\sigma)^2=1+\vartheta.
}
By \eqref{e:easy}, \eqref{eqf0fmu} and $g\in L^\infty(\Omega;\R^{n})$ we obtain
\begin{multline}\label{e:G0G}
G_0(u,\cost,\beta, B_\rho)\le (1+\sigma) G(u,\cost,\beta, B_\rho)
 +{\omega_n}(\parametermu^{\sfrac p2}+{k_\sigma}\|g\|^p_{L^\infty(\Omega;\R^{n})})\rho^{n}
 \\
 \le  (1+2\sigma)G(u,\cost,\beta, B_\rho),
\end{multline}
provided 
\begin{equation}\label{e:R12}
{\omega_n}(\parametermu^{\sfrac p2}+{k_\sigma}\|g\|^p_{L^\infty(\Omega;\R^{n})}) 
R_1^{\sfrac 14}\leq \sigma.
\end{equation}
Now, for any field $v\in GSBD^p(\Omega)$ with {$\{v\ne u\}\subset\hskip-0.125cm\subset B_\rho$}, 
being $u$ a local minimizer of $G$, \eqref{e:easy} and \eqref{eqf0fmu} give
\begin{multline*} 
G(u,\cost,\beta, B_\rho)\leq G(v,\cost,\beta, B_\rho)\\
 \leq  (1+\sigma)G_0(v,\cost,\beta, B_\rho)+ {\omega_n}
 {k_\sigma}({\parametermu^{p/2}}+\|g\|^p_{L^\infty(\Omega;\R^{n})})\rho^{n},
\end{multline*}
as \eqref{e:s3} is violated, we infer
\begin{equation}\label{e:G0G2}
\big(1-{\omega_n} {k_\sigma}({\parametermu^{p/2}}+\|g\|^p_{L^\infty(\Omega;\R^{n})}) R_1^{\sfrac 14}\big)
G(u,\cost,\beta, B_\rho)\leq (1+\sigma)G_0(v,\cost,\beta, B_\rho).
\end{equation}
{We choose $R_1\in(0,R)$ such that  \eqref{e:R11}, \eqref{e:R12} and
\begin{equation}\label{e:R13}
\frac{1+\sigma}{1+2\sigma}\leq {{1-{\omega_n}\, {k_\sigma}({\parametermu^{p/2}}+\|g\|^p_{L^\infty(\Omega;\R^{n})})
 R_1^{\sfrac 14}}}
\end{equation}
are satisfied. Then \eqref{e:G0G2} becomes $G(u,\cost,\beta, B_\rho)\leq (1+2\sigma)G_0(v,\cost,\beta, B_\rho)$, so that 
 recalling \eqref{e:G0G}  and the 
choice of $\sigma\in(0,\sfrac 12)$ made {in \eqref{e:sigma}}, we get}
\[
{G_0}(u,\cost,\beta, B_\rho)\leq (1+\vartheta)G_0(v,\cost,\beta, B_\rho).
\]
We finally deduce \eqref{F0phi0} from the latter inequality by taking the infimum on the
class of admissible $v$ introduced above (cp. the definition of $\Phi_0$ in \eqref{e:phi0}).
\smallskip

\noindent\emph{Step 2.} Fix $R_2>0$ such that 
\be{\label{e:rho2}
R_2<R_1\wedge(\beta\,\eps)^2\wedge \tau,
}
where ${0<R_1\leq R}$ satisfies \eqref{e:R11}, \eqref{e:R12} and \eqref{e:R13}.
For any $\rho<R_2$ set $\rho_i:=\tau^i\rho$, $i\in \N$. Let us show by 
induction that \eqref{e:s0} implies for all $i\in \N$ 
\begin{equation}\label{e:s7}
G(u,\cost,\beta,B_{\rho_i}(x))<\beta\,\eps\rho_i^{{n-1}}.
\end{equation}
The first inductive step $i=0$ is exactly \eqref{e:s0}. 
Suppose now that \eqref{e:s7} holds for some $i$, then by Step~1 either \eqref{e:s1} 
or \eqref{e:s2} holds. In the former case by \eqref{e:rho2} we have
\[
G(u,\cost,\beta,B_{\rho_{i+1}}(x))<\tau^{n-1}\rho_i^{{n-\sfrac12}}=\rho_i^{\sfrac 12}\rho_{i+1}^{{n-1}}<
\beta\,\eps\rho_{i+1}^{{n-1}}.
\]
Instead, in the second instance by the inductive assumption we infer, {since $\tau\le 1$},
\[
G(u,\cost,\beta,B_{\rho_{i+1}}(x))<\tau^{{n-\sfrac12}}G(u,\cost,\beta,B_{\rho_{i}}(x))
{<}\tau^{{n-\sfrac12}}\beta\,\eps\rho_i^{{n-1}}<
\beta\,\eps\rho_{i+1}^{{n-1}}.
\]
\smallskip

\noindent\emph{Step 3.} 
Let $\sigma={\sigma(\tau)}>0$ be as in \eqref{e:sigma}, and fix $R_0>0$ such that 
$R_0\leq\frac{\beta\eps}{{2}{\omega_n {k_\sigma}
({\parametermu^{p/2}}+\|g\|^p_{L^\infty(\Omega;\R^n)})}}
\wedge R_2$, with $R_2$ defined in \eqref{e:rho2}.
 We claim that for all $\rho\in(0,R_0)$
\be{\label{e:s9}
G_0(u,\cost,\beta,B_{\rho}(x))\geq \vartheta_0\,\rho^{{n-1}}
}
with $\vartheta_0:=\frac{\beta\eps}{{2}({1+\sigma})}$. 

By contradiction, if \eqref{e:s9} does not hold, we find by \eqref{e:easy}, \eqref{eqf0fmu}, and since $\rho<R_0$
\begin{multline*}
G(u,\cost,\beta,B_{\rho}(x))\\ \leq (1+\sigma) G_0(u,\cost,\beta,B_{\rho}(x))
+{\omega_n} {k_\sigma}({\parametermu^{p/2}}+\|g\|^p_{L^\infty(\Omega;\R^{n})})\rho^{n}
<{\beta}\eps\rho^{{n-1}}.
\end{multline*}
Hence \eqref{e:s0} holds true,
and therefore by Step 2 inequality \eqref{e:s7} yields 
\[
\liminf_{\rho\to0}\frac{1}{{\rho^{n-1}}}{G}(u,\cost,\beta,B_{\rho}(x))\leq
{\beta}\eps,
\]
in turn implying
\[
\liminf_{\rho\to0} \frac{\calH^{n-1}(J_u\cap B_{\rho}(x))}{{\omega_{n-1}\rho^{n-1}}}\leq {\frac{\eps}{2\omega_{n-1}}<1}
\]
{by the definition of $\eps$, so contradicting \eqref{e:densone}}.  
{This concludes the proof of (\ref{e:s9}) {for points in $J_u^*$}. 

{Finally, since} the definitions of $R_0$ and 
$\vartheta_0$ are independent of the particular point $x\in J_u^\ast$, \eqref{e:s9}
readily extends to {$\Omega\cap \overline{J_u^\ast}$} and  (\ref{e:lb}) is proven.}
\end{proof}

The density lower bound for the jump term of the energy follows straightforwardly.
\begin{corollary}[{Density lower bound for the jump}]\label{c:dlbjump}
{Under the same assumptions as Lemma~\ref{l:lowbound}, there 
exist }
$\vartheta_1$ and $R_1$, depending only on {$n$}, $p$, $\mathbb{C}$, $\cost$, $\beta$, 
$\parametermu$, and $\|g\|_{L^\infty(\Omega;\R^{n})}$, 
such that if $0<\rho<R_1$, {$x\in\Omega\cap \overline{J_u^*}$,} 
and $B_\rho(x)\subset\hskip-0.125cm\subset\Omega$, then 
\be{\label{e:dlbjump}
\calH^{n-1}(J_u\cap B_\rho(x))\geq \vartheta_1\rho^{n-1}.}
\end{corollary}
\begin{proof}
Let {$x\in\Omega\cap \overline{J_u^*}$} and $B_\rho(x)\subset\hskip-0.125cm\subset\Omega$. 
{Denoting by $\vartheta_0$ and $R_0$ the constants in Lemma~\ref{l:lowbound},}
if $\rho\in(0,R_0]$ we have both
\begin{equation}\label{e:stimabasso}
 G_0(u,\cost,\beta,B_\rho(x))\geq\vartheta_0\rho^{n-1}
\end{equation}
by Lemma~\ref{l:lowbound} {itself}, {and the energy upper bound
\[
 G(u,\cost,\beta,B_\rho(x))\leq {2}n\omega_n\beta\rho^{n-1}+\omega_n\cost\|g\|^p_{L^\infty(B_\rho(x);\R^n)}\rho^n.
\]
The latter easily follows by the local minimality of $u$ and comparing its energy with that
of $u\chi_{ B_\rho(x)\setminus B_{\rho-\delta}(x)}$ and then letting $\delta\downarrow 0$.
Moreover, by taking into account the first inequality in 
\eqref{eqf0fmu}, we have that
\[
 G_0(u,\cost,\beta,B_\rho(x))\leq 2^{p-1}G(u,\cost,\beta,B_\rho(x))
 +\omega_n\rho^n\big(\mu^{\sfrac p2}+2^{p-1}\cost\,\|g\|^p_{L^\infty(B_\rho(x);\R^n)}\big).
\]
Hence, {for all} $\rho\in(0,1\wedge R_0]$ we conclude that 
\begin{align}\label{e:easy2}
 G_0(u,\cost,\beta,B_\rho(x))\leq & 
 {2^{p}}n\omega_n\beta\rho^{n-1}\notag
 \\&+\omega_n\big(\mu^{\sfrac p2}+2^{p}\cost\,
 \|g\|^p_{L^\infty(B_\rho(x);\R^n)}\big)\rho^n
 \leq {c_*} \rho^{n-1},
\end{align}
where ${c_*}$ depends on  $n$, $p$, $\cost$, $\beta$, $\parametermu$, and 
$\|g\|_{L^\infty(\Omega;\R^n)}$.}

{We fix $\gamma\in(0,1)$, for example $\gamma=1/4$ as above, and choose}
$\tau\in(0,\tau_\gamma]$ 
{in the decay property} such that ${c_*}\,\tau^{1-\gamma}<\vartheta_0$.
{Let  $\varepsilon=\varepsilon(\tau)>0$, $\vartheta=\vartheta(\tau)$  and 
$R=R(\tau)>0$ be the constants provided by  the decay property. We now show that}
\begin{equation}\label{eq:stimasalbasso}
\calH^{n-1}(J_u\cap B_\rho(x))>\varepsilon \rho^{n-1}
\end{equation}
for all $\rho\in(0,R_1]$, with $R_1:=1\wedge R_0\wedge R$.
Indeed, \text{arguing as in \eqref{e:sigma}-\eqref{e:R13}} from  \eqref{e:stimabasso} we deduce that for $\rho\leq R_1$ 
\[
G_0(u,\cost,\beta,B_{{\rho}}(x))\leq(1+\vartheta)\Phi_0(u,\cost,\beta, B_{{\rho}}(x)).
\]
If \eqref{eq:stimasalbasso} were false we would conclude {using (\ref{e:stimabasso}), the decay property} {and \eqref{e:easy2}}
for some $\bar{\rho}\in(0,R_1]$ that
\begin{align*}
 \vartheta_0{(\tau \bar{\rho})^{n-1}} &\leq G_0(u,\cost,\beta,B_{\tau \bar{\rho}}(x))\\
 &\leq {\tau^{n-\gamma}}G_0(u,\cost,\beta,B_{\bar{\rho}}(x))
 \leq { c_*\,\tau^{n-\gamma}}\,\overline{\rho}^{n-1},
\end{align*}
contradicting the choice of $\tau$.
\end{proof}

\begin{corollary}\label{lemmaomegauopen}
{Under the same assumptions as  Lemma~\ref{l:lowbound},
the set
\[
\Omega_u:=\{x\in\Omega:\,G_0(u,\cost,\beta,B_\rho(x))<\vartheta_0\rho^{n-1}\quad\text{for some $\rho\in(0,R_0 {\wedge \dist(x,\partial\Omega)})$}\}                                
\]
is open and obeys
 $\Omega_u\cap \overline{J_u^\ast}=\emptyset$. {Moreover,
 $\calH^{n-1}(\Omega_u\cap J_u)=0$.}}
\end{corollary}
\begin{proof}
Let $x\in \Omega_u$. Then there is $\rho\in(0,R_0)$ with $G_0(u,\cost,\beta,B_\rho(x))<\vartheta_0\rho^{n-1}$, and therefore
there is $\delta\in(0,\rho)$ such that 
\[
 G_0(u,\cost,\beta,B_\rho(x))<\vartheta_0(\rho-\delta)^{n-1}.
\]
The inclusion $B_\delta(x)\subset\Omega_u$ follows straightforwardly. Indeed, let $y\in B_\delta(x)$, we have
\begin{equation*}
 G_0(u,\cost,\beta,B_{\rho-\delta}(y))\leq G_0(u,\cost,\beta,B_\rho(x))
 <\vartheta_0(\rho-\delta)^{n-1}.
\end{equation*}
Therefore $\Omega_u$ is open.

 By Lemma~\ref{l:lowbound} and the definition we immediately obtain 
 $\Omega_u\cap \overline{J_u^\ast}{=\emptyset}$.
 
Since $\calH^{n-1}(J_u\setminus J_u^\ast)=0$, by the ${(n-1)}$-rectifiability 
of $J_u$, and $\Omega_u\cap J_u^\ast=\emptyset$,
we infer that $\calH^{n-1}(\Omega_u\cap J_u)=0$.
\end{proof}
{In dimension $2$ the assumptions of Lemma~\ref{l:lowbound} hold true {and Sobolev minimizers are regular everywhere},
therefore we may conclude the following result.}
{\begin{theorem}\label{t:theogsbdpmin}
 Let $n=2$, $\Omega\subset\R^2$ open, $p\in(1,\infty)$, $\kappa\ge0$, 
 $\beta>0$, $\mu\ge0$, {$g\in L^\infty(\Omega;\R^2)$ if $p\in(1,2]$
 and $g\in W^{1,p}(\Omega;\R^2)$ if $p>2$.}

Let $u\in GSBD^p(\Omega)$ be a local minimizer of $G$ {according to \eqref{e:Glocmin}, then} $\Omega\cap\overline{S_u}=
\Omega\cap\overline{J_u}=\Omega\setminus\Omega_u$,
 \be{\label{e:saltochiuso}
{\calH^{1}(\Omega\cap{\overline{J_u}\setminus J_u})=0}
}
{and $u\in 
C^1(\Omega\setminus\overline{J_u};\R^2)$.} 
\end{theorem}
\begin{proof}
Since $GSBD^p$ is defined via slices {and $\Omega_u$ is open}, 
from $\calH^{1}(\Omega_u\cap J_u)=0$ we deduce  
$u\in W^{1,p}_\loc(\Omega_u;\R^2)$.
Thus, by elliptic regularity of Theorem~\ref{t:regularity}  we obtain 
$u\in C^{1}(\Omega_u;\R^2)$.
Hence, {$S_u\subseteq\Omega\setminus\Omega_u$ and actually 
$\Omega\cap\overline{S_u}\subseteq\Omega\setminus\Omega_u$, 
as $\Omega_u$ is open.}

On the other hand, if $x\in \Omega\setminus\overline{J_u}$, then $u\in W^{1,p}(B_\rho(x);\R^2)$ for some 
$\rho>0$, as $GSBD^p$ is defined via slices and again by elliptic regularity $u\in C^{1}(B_\rho(x);\R^2)$. 
Thus, $x\in\Omega_u$, and {since $J_u\subseteq S_u$} we conclude $\Omega\setminus\Omega_u={\Omega\cap\overline{S_u}=\Omega\cap\overline{J_u}}$.

Eventually, \eqref{e:saltochiuso} is a straightforward consequence of \eqref{e:dlbjump} and 
\cite[Theorem 2.56]{AmbrosioFuscoPallara}. 
\end{proof}
}

\subsection{Proof of the main results}\label{mainproof}

{We are finally ready to establish existence of strong minimizers for the Griffith 
static fracture model.
For simplicity of notation we write the functional $G$ appearing in \eqref{e:G} as $G(\cdot)=G(\cdot,\cost,\beta,\Omega)$.}
\begin{proof}[Proof of Theorem \ref{theootherp}]

By the compactness and lower semicontinuity result \cite[Theorem 11.3]{gbd}, $G$ has a minimizer $u$ in $GSBD(\Omega)$.
{By Theorem \ref{t:theogsbdpmin} 
we obtain  $u\in C^1(\Omega\setminus\overline{J_u};{\R^2})$ {so that}
 $E_p(\overline{J_u},u)=G(u)$, $E_p$ being defined in \eqref{eqgriffintrop}. }
Now, if $\Gamma\subset\overline\Omega$ is closed and $v\in W^{1,p}_\loc(\Omega\setminus \Gamma;{\R^2})$ with $E_p(\Gamma,v)<\infty$, then $v\in GSBD(\Omega)$ with ${\calH^1}(J_v\setminus \Gamma)=0$, again arguing by slicing. We conclude that
$$E_p(\overline{J_u},u)=G(u)\leq G(v)\leq E_p(\Gamma,v).$$
\end{proof}
The proof of Theorem \ref{theop2} is analogous.

\section*{Acknowledgments} 
This work was partially supported 
by the Deutsche Forschungsgemeinschaft through the Sonderforschungsbereich 1060 
{\sl ``The mathematics of emergent effects''}.
S.~Conti thanks the University of Florence for
the warm hospitality of the DiMaI ``Ulisse Dini'', where part of this work was 
carried out. 
M.~Focardi and F.~Iurlano are members of the Gruppo Nazionale per
l'Analisi Matematica, la Probabilit\`a e le loro Applicazioni (GNAMPA)
of the Istituto Nazionale di Alta Matematica (INdAM).


\end{document}